\documentclass[letterpaper]{article}
\usepackage{aaai} \usepackage{times} \usepackage{helvet} \usepackage{courier}
\usepackage{amssymb, amsmath, amsthm}
\nocopyright

\newtheorem{theorem}{Theorem}[section]
\newtheorem{lemma}[theorem]{Lemma}
\newtheorem{proposition}[theorem]{Proposition}
\newtheorem{corollary}{Corollary}

\theoremstyle{definition}

\newtheorem{example}[theorem]{Example}

\newcommand*{\N}{\ensuremath{\mathbb{Z}}}

\renewcommand{\phi}{\varphi}

\date{November 18, 2011}

\title{A Cost-Minimizing Algorithm for School Choice
\thanks{Aksoy, Azzam, Coppersmith and Karaali were partially supported by NSF Grant DMS-0755540. Karaali was partially supported by a Pomona College Hirsch Research Initiation Grant and a National Security Agency Young Investigator Award (NSA Grant $\#$H98230-11-1-0186). Zhao was partially supported by the Hutchcroft Fund of the Department of Mathematics and Statistics at Mount Holyoke College. Zhu was partially supported by a Mount Holyoke College Ellen P. Reese Fellowship.}\thanks{This note is an abbreviated and revised version of a research article with the same title which was posted on the arXiV in October 2010. The original article included a proposal for a second criterion, a modification of the Index-Based Hungarian Mechanism to optimize this second criterion, detailed strategic analyses, requisite proofs, and several more examples. (See \texttt{http://arxiv.org/abs/1010.2312} version 1.) This current version was prepared to be presented as an invited paper at the ISAIM 2012 (International Symposium on Artificial Intelligence and Mathematics, Fort Lauderdale, FL. January 9--11, 2012) special session on computational social choice. 
After the conference the authors decided to change the direction of their project. We see this note as a record of our earlier thoughts in this direction. To the best of our knowledge, this version contains no incorrect assertions.}}

\author{S. Aksoy \\ University of Chicago \\ Chicago, IL \And
A. Azzam \\ University of Nebraska Lincoln \\Lincoln, NE \And
C. Coppersmith \\ Bryn Mawr College \\ Bryn Mawr, PA \And
J. Glass \\ California State University\\ East Bay, Hayward, CA \AND
G. Karaali \\Pomona College \\Claremont, CA\And 
X. Zhao \\ Mount Holyoke College\\  South Hadley, MA\And
X. Zhu \\ Mount Holyoke College\\ South Hadley, MA}

\begin{document}

\maketitle

\begin{abstract}
The school choice problem concerns the design and implementation of matching mechanisms that produce school assignments for students within a given public school district. In this note we define a simple student-optimal criterion that is not met by any previously employed mechanism in the school choice literature. We then use this criterion to adapt a well-known combinatorial optimization technique (Hungarian algorithm) to the school choice problem. 
\end{abstract}

\section{Introduction}
\label{S:Introduction}

School choice policies are processes by which families have some say in determining where their children go to school.  Since the late eighties such policies have been adopted by many school districts across the nation.  Before school choice, students were typically assigned to public schools according to proximity.  Since wealthy families have the means to move to areas with desirable or reputable schools, such families have always had \emph{de facto} school choice.  Children in families that could not afford such a privilege were left with no other option than to attend the closest school - whether or not the school was desirable and/or was a good fit.  Thus school choice has been celebrated as a successful tool giving more families the power to shape their children's education, regardless of socioeconomic background.

In many school districts where funding and experienced teachers are lacking, school quality is uneven and often a small number of schools are strongly preferred over others. Since it is not possible to assign all students to their top choice school, the question of \emph{how} to assign students to schools is often regarded as the central issue in school choice. In order to safeguard parents who seek to have their children attend schools conveniently within walking distance, at which a sibling is enrolled, or those offering need-based programs, districts define and adhere to a handful of school {\it priorities} which encapsulate such constraints. Thus school choice can be viewed as a
two-sided matching problem. An extensive study of two-sided matching problems can be found in \cite{RoSo90}; a more recent historical overview is \cite{Ro08}.

Previous work on school choice as a matching problem evaluates assignments using the notions of stability, Pareto efficiency and strategyproofness. Though all worthy considerations, these do not necessarily suffice to promote the most desirable outcomes.
In the context of school choice, stability corresponds to preventing priority violations. A priority violation occurs when a student desires a school more than the school to which she was assigned, and has higher priority than a student assigned to her desired school. Preventing priority violations is desirable for a very pragmatic reason:
Students whose priorities are violated may have legitimate grounds for legal action. Even without legal recourse, it is often felt that students are ``entitled" to schools in which they have been prioritized.
However the focus on avoiding priority violations in current school choice mechanisms leads to documented inefficiencies.
See \cite{AbPaRo09}, \cite{ErEr08}, \cite{Ke10}, \cite{Ro82} for more on this potential tradeoff between stability and efficiency.

In this note, we propose an approach to the school choice problem which focuses on student preferences rather than school priorities.

Our preference index naturally associates a ``cost" to each matching and as a result we conceptualize the school choice problem as a ``cost minimizing" assignment problem. In this context, a lower cost assignment corresponds to a matching that in some sense more closely meets student preferences. We can then introduce a new mechanism, adapted from a well-known combinatorial optimization algorithm, that produces a matching minimizing this cost, i.e., meeting student preferences as closely as possible. With respect to our index, the assignments produced by this mechanism often outperform the outcomes of standard mechanisms used or proposed in recent literature. Furthermore our method may be modified to respect and utilize student preference compatibilities. 
In cities without well-defined or legally required priorities (e.g. those that use whole-city lotteries), our approach provides policy makers two possible ways to create the most optimal student matching. Even cities committed to respecting student priorities may find these ideas valuable as priorities may indeed be incorporated at an intermediate or a final stage, see the relevant discussion in \S\ref{S:Conclusion}.


\subsection{Research background}
\label{SS:LitReview}

School district policy decisions have long provided active lines of inquiry for public policy designers, operations researchers, economists and education administrators. Much of the relevant work has focused on designing school district boundaries in order to optimize various measures. For a diverse yet representative selection of work in this vein, see \cite{BrKn05}, \cite{CaShGuWe04}, \cite{FeGu90}, \cite{FrKo73}.

In our work we focus on assignment policy as a mechanism design problem, which provides a natural framework to investigate means of implementing social goals (cf.\ \cite{Ma08}). In the following we will refer to three specific mechanisms. The first two were introduced in \cite{AbSo03} while the third was presented in \cite{Ke10}.
\begin{enumerate}
\item Student-Optimal Stable Matching Mechanism (SOSM)
\item Top Trading Cycles Mechanism (TTC)
\item Efficiency Adjusted Deferred Acceptance Mechanism (EADAM)
\end{enumerate}

SOSM adapts the famous Gale-Shapley Deferred Acceptance (DA) algorithm \cite{GaSh62} to the school choice problem. It is well established as a stable and strategyproof mechanism that has already been implemented in several large urban school districts \cite{AbPaRo09}, \cite{AbPaRoSo06}. However, when applied to large-scale data SOSM may lead to some welfare losses \cite{Ke10}. TTC is an alternative mechanism which promotes efficiency as opposed to stability, and is also strategyproof. The basic algorithm is to create trading cycles alternating between students and schools and to allow efficient matchings.
EADAM is proposed in \cite{Ke10} as a way to alleviate some of the efficiency costs of stability by iteratively running SOSM and modifying the preferences of any interrupters (who are students who cause others to be rejected from a school which later on rejects them) such that the SOSM outcome is Pareto dominated. As any Pareto domination of SOSM will lead to priority violations (cf.\ \cite{GaSh62}), EADAM leads to at least one priority violation. We will not need the specific processes in our work.

Recent literature also examines various real-life mechanisms such as those from Boston \cite{APRS05}, Chicago \cite{CuJaLe06}, Milwaukee \cite{GrPeDu99}, \cite{Ro98}, and New York City \cite{APR05}.


\subsection{Notation, basic terms and our model}
\label{SS:Notation}

Let $I$ denote a nonempty set of students, and $S$ a nonempty set of schools. A {\bf matching} $M: I\to I\times S$ is a function that associates every student with exactly one school, or potentially no school at all. Write $\mathfrak{M}$ for the set of matchings. 
We write $M_i = s$ if $M(i) = (i,s)$.
For all $s\in S$, we let $q_{s}$ denote the {\bf capacity} of $s$.

A {\bf preference profile} for a student $i \in I$, written ${P}_i$,
is a tuple $(S_{1},\dots,S_{n})$ where the $S_{j}$'s form a partition of $S$ and every element of $S_{j}$ is preferred to every element of $S_{k}$ if and only if $j<k$.
Define the {\bf ranking function} $\phi_i : S \rightarrow \N$  of a student $i \in I$ by letting $\phi_i(s)$ denote $i$'s {ranking} of $s \in S$. In other words $\phi_i(s) = j$ if $s \in S_j$. When each $S_{j}$ is singleton, we say that $i$'s preference profile is {\bf strict}, (in which case we can view $P_i$ as an $n$-vector).  If $s_{k},s_{l}\in S_{j}$ for some $j,k\neq l$, then we say that the student is {\bf indifferent} between $s_{k}$ and $s_{l}$. If $i$ prefers $s_{k}$ to $s_{l}$, we write $s_k \succ_i s_l$, or simply $s_k \succ s_l$ if $i$ is unambiguous. We denote a set consisting of preference profiles for each student in $I$ by ${\bf P} = \{P_i : i \in I\}$ and the space of all such sets is denoted by $\mathfrak{P}$.

A {\bf priority structure} for a school $s \in S$, written ${\Pi}_s$,
is a tuple $(I_{1},\dots,I_{n})$ where the $I_{j}$'s form a partition of $I$ and every element of $I_{j}$ is preferred to every element of $I_{k}$ if and only if $j < k$. When each $I_{j}$ is singleton, we say that $s$'s priority structure is {\bf strict}, (in which case we can view $\Pi_s$ as an $n$-vector).  If $i_{k},i_{l}\in I_{j}$ for some $j,k\neq l$, then we say that the school is {\bf indifferent} between $i_{k}$ and $i_{l}$. If $s$ prefers $i_{k}$ to $i_{l}$ we write
$i_k \succ_s i_l$, or simply $i_k \succ i_l$ if $s$ is unambiguous from context. We denote a set consisting of priority structures for each school in $S$ by ${\bf \Pi} = \{\Pi_s : s \in S\}$ and the space of all such sets is denoted by $\mathfrak{\Pi}$.

A matching $M'$ {\bf (Pareto) dominates} $M$ if $M'_i \succ_i M_i$ for all $i$ and $M'_j \succ_j M_j$ is strict for some $j$. A {\bf (Pareto) efficient matching} is a matching that is not (Pareto) dominated.

A {\bf matching mechanism} $\mathcal{M}: \mathfrak{P} \times \mathfrak{\Pi} \to \mathfrak{M}$
is a function that takes an ordered pair $(\textbf{P}, {\bf \Pi}) \in \mathfrak{P} \times \mathfrak{\Pi}$ of preferences and priorities and produces a matching.

Let $\Pi_{s}$ be a priority structure for school $s$. We say that a matching $M$ {\bf violates the priority} of $i\in I$ for $s$ if there exist some $j\in I$ and $s^{\prime} \in S$ such that
\begin{enumerate}
\item $M_j=s$, $M_i = s^{\prime}$: $j$ gets assigned $s$ under $M$ and $i$ gets assigned $s^{\prime}$ under $M$.
\item $s \succ_i s^{\prime}$: $i$ prefers attending $s$ over $s^{\prime}$;  and
\item  $i \succ_s j$: $s$ prioritizes $i$ over $j$.
\end{enumerate}
We say that a matching $M$ is {\bf stable} if
\begin{enumerate}
\item $M$ does not violate any priorities.
\item No student is matched to a lower-ranked school when a more preferred school is unfilled.
\end{enumerate}
A {\bf stable mechanism} is a mechanism that always produces stable matchings.
A mechanism is \textbf{strategyproof} if no student can ever receive a more preferred school by submitting falsified, as opposed to truthful, preferences.

In the above context, the goal of the school choice problem is to find a matching mechanism $\mathcal{M}$ which satisfies certain criteria.


\section{A Natural criterion for evaluating assignment mechanisms}
\label{S:NewCriteria}

The current literature on school choice proposes mechanisms that balance student preferences and school priorities by using stability, (Pareto) efficiency, and strategyproofness as the standard criteria for evaluating the desirability of a given mechanism. In our work, we emphasize student preferences. With this spirit, we introduce a new student-optimal criterion to expand the scope of what constitutes a good mechanism.

The question of what criteria to use to judge the quality or desirability of a mechanism is a difficult one; for example, see \cite{Mc09} where McFadden argues that tolerance of behavioral faults should be included in such a list of criteria. The goal of school districts when designing a school choice policy is not singular (unlike, for instance, the case of auction design where our sole objective is to maximize selling price). Thus, it is especially important for us to define a yardstick by which we measure the success of a given school choice mechanism. We could technically define the best school mechanism as one that minimizes the government education funding budgets, produces the most elite students, or improves the conditions of less-advantaged students the most, etc. Obviously, the ultimate design depends on how we define the objective/criteria of the school choice problem.

In this section we propose a simple criterion which provides a possible interpretation of how to best honor student preferences.
Our index 
inherently captures the utilitarian objective of improving the outcome for as many students as possible.


\subsection{A preference reverence index}
\label{SS:Index}

We begin with an example from \cite{Ro82} which illustrates the possible tradeoff between stability and efficiency mentioned in \S\ref{S:Introduction}.
Assume there are three schools, $s_{1}, s_{2}, s_{3}$ and three students $i_{1}, i_{2}, i_{3}$. The priorities of the schools and the preferences of the students are given by:
\[
\begin{array}{ccc} s_{1}: i_{1} \succ i_{3} \succ i_{2}& \qquad & i_{1}:s_{2} \succ s_{1} \succ s_{3}\\ s_{2}:i_{2} \succ i_{1} \succ i_{3}& \qquad & i_{2}: s_{1} \succ s_{2} \succ s_{3}\\ s_{3}:i_{2} \succ i_{1} \succ i_{3}& \qquad & i_{3}: s_{1} \succ s_{2} \succ s_{3} \end{array}\]
Here, the only stable matching is:
\[
\begin{pmatrix}i_{1}& i_{2}& i_{3}\\ s_{1}&s_{2}&s_{3}\end{pmatrix}.\]
But this matching is (Pareto) dominated by:
\[
\begin{pmatrix}i_{1}& i_{2}& i_{3}\\ s_{2}&s_{1}&s_{3}\end{pmatrix}.\]
We see that the second matching (Pareto) dominates the first matching because it gives $i_1$ and $i_2$ schools they preferred over their stable matching. However, this (Pareto) efficient matching is not stable because $i_2$ now violates $i_3$'s priority in $s_1$.

Now, we construct a ``preference reverence index" that, on the students' side, would quantify how much their preferences were dismissed, and a ``priority reverence index" that, on the schools' side, would quantify how much their priorities were dismissed. We assign numbers to the preferences and priorities of schools. Namely, we assign a 0 to a student's first choice, a 1 to their second choice, and so on. For the stable matching we have:
\begin{center}
Schools: 0 + 0 + 2 = 2 (priority reverence index)

Students: 1+ 1+ 2 = 4 (preference reverence index)
\end{center}

And, for the (Pareto) efficient matching we have:
\begin{center}
Schools: 1 + 2 + 2 = 5  (priority reverence index)

Students: 0 + 0 + 2 = 2  (preference reverence index)
\end{center}

We see that the (Pareto) efficient matching has better served the student preferences (lower preference reverence index) at the expense of the school priorities (higher priority reverence index). In fact, by the definition of (Pareto) efficiency, all (Pareto) dominations of stable matchings under SOSM improve how well student preferences are honored, thereby lowering the preference index.

Let us make the above precise.

Let $I$ be a nonempty set of students, and $S$ be a nonempty set of $m$ schools. Recall that for any $i\in I$, $s\in S$, $\phi_i(s)$ is $i$'s ranking of $s$ and for any matching $M : I \rightarrow I \times S$, $M_i = s$ denotes that $M(i) = (i,s)$.
Let $\mathfrak{M}$ be the set of matchings.
Define $\mu: \mathfrak{M} \to \N$ by
\[
\mu(M)=\sum_{i \in I}\left ( \phi_i(M_i)-1 \right ) .\]
For any given $M \in \mathfrak{M}$ we will call $\mu(M)$ the {\bf preference reverence index} of $M$ or simply the {\bf preference index}.

For a given set of schools and students each equipped with priorities and preferences respectively, if there exists a matching $M$ such that $\mu(M)=0$, we will say that the students are {\bf preference compatible}.  If there exists no such matching, the students are {\bf preference incompatible}. In \S\S\ref{SS:RankCompatibility} we will refine this notion further.

Since $\mathfrak{M}$ is finite, $\mu(\mathfrak{M})$ is finite and hence there exists some $M\in \mathfrak{M}$ such that $\mu(M)\le \mu(M')$ for all $M'\in \mathfrak{M}$. We will describe a method of seeking and locating such a minimal index matching in \S\ref{S:HA}. Here we focus on some properties of our new index.

First we investigate how our index relates to the standard notion of (Pareto) efficiency. Here is a useful lemma:
\begin{lemma}
Let $M$, $M^{\prime} : I \rightarrow I \times S$ be two matchings. If $M$ (Pareto) dominates $M^{\prime}$, then $\mu(M) < \mu(M^{\prime})$.
\end{lemma}

It clearly follows that our index is consistent with the more standard notion of (Pareto) efficiency. However, the index can also distinguish between two Pareto incomparables, and thus provide a way to define different levels of efficiency. Consider for instance the following:

\begin{example} Assume we have three students and three schools, and each school has one spot. The student preferences over the schools are given as:
\[
\begin{array}{cc} i_{1}:s_{1} \succ s_{2} \succ s_{3}\\ i_{2}: s_{3} \succ s_{2} \succ s_{1}\\ i_{3}:s_{3} \succ s_{1} \succ s_{2}\end{array}
\]
Here are two possible Pareto efficient matchings:
\[ \text{Pareto Efficient Matching } \#1 \quad (\mu=1): 
\begin{pmatrix}i_{1}& i_{2}&i_{3}\\ s_{1}&s_{2}&s_{3}\end{pmatrix}\]
\[ \text{Pareto Efficient Matching } \#2 \quad (\mu=2): 
\begin{pmatrix}i_{1}& i_{2}&i_{3}\\ s_{1}&s_{3}&s_{2}\end{pmatrix}\]

Both solutions are (Pareto) efficient because in each, there is no way to make any of the students better off without making another student worse off. However, they are (Pareto) incomparable as moving from one to the other would necessarily harm at least one student. When we compare the two in terms of their preference indices, we see that one has a lower index than the other, and hence in terms of the index, Matching 1 is preferable to Matching 2. In other words, all Pareto efficient matchings are \textit{not} made equal, and the preference index provides a way in which we can differentiate between (Pareto) efficient matchings. We can thus have some basis by which to rate some ``better" while rating others ``worse" provided that we are only given a list of ordinal preferences.
\end{example}

Next we focus on our index and its relation to stability. We begin with a result that follows from the above Lemma:
\begin{corollary}
Let $\mathcal{M}$ be a stable mechanism and let $\mathcal{DA}$ be the deferred acceptance algorithm. Then for any preference profile, $\mu(\mathcal{DA}(P))  \le \mu(\mathcal{M}(P))$.
\end{corollary}

There can exist two stable matchings with the same preference index. 
However, there is only one lowest preference index stable matching:
\begin{lemma}
\label{L:UniqueStableMin}
The SOSM matching is the unique lowest preference index stable matching for a given preference profile. In other words the inequality in the above Corollary is strict.
\end{lemma}

The preference index measures how well ordinal preferences are being honored as a whole. Each time we move to the next-best choice in a student's ranking, this counts as ``1 violation" of their preferences, and we then add up the number of times we make such violations. Thus, perhaps a more apt title would be ``preference dismissal index" since it is a measure of how little the preferences are being ``honored" or ``revered." It should be noted that the preference index assumes that it is the same to give one student their fifth choice and one their first choice (Total=4) as it is to give two students their third choice (Total=4).


\section{A cost-minimizing school choice mechanism}
\label{S:HA}

In \S\ref{S:NewCriteria} we defined a new notion, the preference index. We then showed that previously proposed school choice mechanisms do not perform optimally with respect to this criterion. In this section we describe an alternative mechanism geared specifically toward this notion. Our mechanism happens to have several other desirable properties which we elaborate on in the latter parts of this section. 

The work in this section is built upon a combinatorial optimization algorithm known as the Hungarian algorithm. The Hungarian algorithm is traditionally used to find the minimum cost matching in various {\bf min-cost max-flow problems} such as assigning individuals to tasks or determining minimum cost networks in travel \cite{Ku55}, \cite{Ku56}. We note that the algorithm can be processed in polynomial time \cite{Mu57}, hence the mechanism itself can effectively be implemented via a computer program.

As the purpose of the Hungarian algorithm is to find the minimum cost matching, the first step in adapting the algorithm to the school choice problem is to define the cost of any particular matching. But we already have a natural candidate. Indeed the preference reverence index defined in \S\ref{S:NewCriteria}
proves to be a good indicator of the cost of a given matching by measuring the cost in terms of the number and extent of preference violations. Thus the resulting mechanism naturally minimizes the preference reverence index.


\subsection{Description}
\label{SS:Description}

In the following we present an elementary description of the Hungarian algorithm equivalent to the original development in \cite{Ku55}. 
For a more sophisticated discussion including computational complexity concerns and an exhaustive investigation of the many variants of the method that lead to impressive complexity improvements, see \cite[Ch.17]{Sc03}.

Let $I$ and $S$ be a set of students and schools, respectively, and assume that a student preference profile ${\bf P}$ is given. Suppose the students are preference incompatible. Since the space $\mathfrak{M}$ of all matchings is finite, $ \mu(\mathfrak{M})$ is finite and therefore there exists some $M\in \mathfrak{M}$ such that $ \mu(M)\le \mu(M')$ for all $ M'\in \mathfrak{M}$. We would like to find such a minimal index matching. In order to do that we will define an associated cost minimization problem.

Let $A =(a_{jk})$ be the $n\times m$ matrix such that $a_{jk}=\phi_{i_{j}}(s_k)$, encoding student preferences. Subtract $1$ from each of the entries to obtain a {\bf cost matrix} $C$, a matrix with no negative entries. 
For now assume that $n=m$, i.e., there is an equal number of students and schools and each school has a capacity of one.
For example for the following preference profile of three students for three schools:
\[
\begin{array}{ccc}
i_1: s_1 \succ s_2 \succ s_3\\
i_2: s_3 \succ s_2 \succ s_1 \\
i_3: s_2 \succ s_3 \succ s_1
\end{array}\]
the matrix $A$ of preferences would be:
\[\begin{array}{c| ccc} & s_{1} & s_{2} & s_{3}\\ \hline i_{1}& 1 & 2& 3\\ \hline i_{2}& 3 & 2 & 1\\ \hline i_{3} & 3 & 1 & 2\end{array}\]
and the associated cost matrix would be:
\[ C = \begin{pmatrix}0&1&2\\
2&1&0\\
2&0&1
\end{pmatrix}.\]
Now the assignment problem reduces to: \emph{Given a cost matrix $C$, pick one element each from each row and each column such that the sum of the selected entries is minimal}. The Hungarian algorithm will then run as follows (cf.\ \cite[Figure 6.1]{RoAn84}):
\begin{enumerate}
\item  Subtract the smallest entry in each row from each entry in that row. [After this stage, all rows have at least one zero entry, and all matrix entries are nonnegative.]
\item  Subtract the smallest entry in each column from each entry in that column. [After this stage, all rows and columns have at least one zero entry, and matrix entries are still nonnegative.]
\item Draw lines through appropriate rows and columns so that all the zero entries of the cost matrix are covered and the minimum number of such lines is used. [There may be several ways to do this, but the main point is that it can be done.]
\item \textbf{Test for optimality:} If the number of covering lines is $n$, then an optimal assignment of all zeroes is possible and we are done; the algorithm terminates. Otherwise, such an assignment is not yet possible, and we proceed to Step 5. 
\item Determine the smallest entry not covered by any line, subtract it from all uncovered entries and add it to all entries covered by both a horizontal and a vertical line. Return to Step 3.
\end{enumerate}
Here is the outcome of the Hungarian algorithm for the preference profile above:
\[\begin{array}{c| ccc} & s_{1} & s_{2} & s_{3}\\ \hline i_{1}& \fbox{1} & 2& 3\\ \hline i_{2}& 3 & 2 & \fbox{1}\\ \hline i_{3} & 3 & \fbox{1} & 2\end{array}\]

We note that Step 5 crucially depends on the following 
\begin{theorem}[Theorem 6.1 \cite{RoAn84}]
If a number is added to or subtracted from all of the entries of any row or column of a cost matrix, then an optimal (minimum cost) assignment for the resulting cost matrix is also an optimal assignment for the original cost matrix.
\end{theorem}

Since the matching algorithm produces the least cost matching \cite{Mu57}, and our cost is represented by the preference index, we see that the resultant matching has the smallest preference index with respect to each student's preferences.
\begin{proposition}
Given a set of preferences, let $M$ be the matching produced by the Hungarian algorithm and let $M'$ be any other matching. Then $\mu(M)\le \mu(M')$.
\end{proposition}


\subsection{A ``Hungarian" school choice mechanism}
\label{SS:Modifications}

In adapting the Hungarian algorithm to the school choice problem, we must make three key modifications. The construction of the algorithm as we presented it above requires as input an $n \times n$ matrix of non-negative numbers,
and it selects as output a unique entry in each row and each column. We must modify the algorithm to accommodate school capacities, unequal numbers of seats and students, as well as preferences containing different numbers of ranked schools.

Assume columns represent schools and rows represent students in our matrix. To express school capacities, we simply add an extra column for each available seat at a school and enter the same preferences for that column.

The third modification addresses the problem of families submitting incomplete preference profiles. It is not immediately obvious why a district should request that students submit lists of the same length. Furthermore, even if a school district requires students to rank a specific number of schools, there will undoubtedly be students who do not list the required number. Regardless, in order to run the Hungarian algorithm, it is necessary to devise a way of completing student preferences such that each student preference list assigns a rank to each school or seat.

A potential solution is to use dummy variables to complete any missing entries in the matrix. However, this method may invite students to strategize.
Even without complete information, students might be motivated to strategize by only submitting their first choice school, thereby weighting this choice with dummy variables so that the algorithm is more likely to select it.

Alternatively we can fill out the remainder of a student preference profile with an equal ranking for all unranked schools. More specifically if a student's preference profile contains only $r$ ranks, then we assign the rank $r+1$ to all the remaining schools. This incentivizes the completion of preference lists, since otherwise all remaining schools will be treated equally.  For instance, if a family puts only their first choice, all other choices will be considered ``second"; therefore they may get a school which they consider terrible at low cost as measured by the mechanism.  Thus it would behoove them to fill out as many schools as possible if they had a genuine preference for one over another.  This method does not take into account priorities, but it does deflect strategizing. We offer it here as a simpler alternative.


\subsection{Properties}
\label{SS:Properties}

Since the Index-Based Hungarian Mechanism ($\mathcal{HM}_i$) selects the ``least cost" matching, when we use preferences as ``costs" we have the following result:
\begin{theorem}
\label{T:MinIndex}
Given a preference profile \textbf{P}, for any mechanism $\mathcal{M}$ \[\mu(\mathcal{HM}_i({\bf P}))\le \mu(\mathcal{M}({\bf P})).\]
\end{theorem}
The Index-Based Hungarian Mechanism meets our utilitarian standard of student optimal matchings. Furthermore, it satisfies the standard notion of efficiency:
\begin{corollary} Given a set of preferences \textbf{P}, let $\mathcal{HM}_i(\textbf{P})$ be the matching produced by the Index-Based Hungarian Mechanism. Then $\mathcal{HM}_i(\textbf{P})$ is (Pareto) efficient.
\end{corollary}

Next we briefly examine the performance of $\mathcal{HM}_i$ with respect to strategic action. Here our exposition is much abridged; the reader is referred to our full research paper for a complete strategy analysis. 

We first assert: 

\begin{proposition}
The $\mathcal{HM}_i$, under complete information, is not strategyproof.
\end{proposition}

The proof involves a concrete construction of a a preference profile with which it is possible for a student to strategize in order to improve her outcome.
Thus under complete information, the Index-Based Hungarian Mechanism is not always immune to strategic action.

However, we are more interested in its performance under incomplete information. In this latter situation, the question is not whether it is \emph{possible} for a student to receive a preferred school by lying, but whether any student has a \emph{rational} motive to falsify preferences. Our natural assumption is that in the macro school choice environment, players are not privy to everyone else's preferences. 

Unfortunately it turns out that there exist preference profiles in which one can strategize under incomplete information in $\mathcal{HM}_i$. Thus our mechanism is not necessarily strategyproof under incomplete information.
Thus $\mathcal{HM}_i$ is Pareto efficient and minimizes the preference index, but it is NOT strategyproof, even with the assumption of incomplete information. (In the full paper we offer possible ways to make up for this weakness.) It is also obviously not stable.

We observed earlier that the SOSM outcome is the minimum preference index stable matching. 
Thus in our framework, we can identify the SOSM outcome as the solution to the school choice problem that minimizes the preference index subject to the constraint of stability. Viewed this way, the $\mathcal{HM}_i$ outcome is the solution to the school choice problem that minimizes the preference index with no constraint. Then, one may ask if making everyone best off relative to each other is ``worth" potentially high numbers of priority violations. The cost of inefficiency imposed by stability has been studied extensively in the literature, but the cost of instability imposed by efficiency has been modeled and studied much less often (but also see \cite{Ta02} for a discussion of a similar trade-off between efficiency and equity).

We can offer a remedy to those who prefer fewer priority violations to more and who are willing to sacrifice some efficiency. With respect to the preference index, we showed that the $\mathcal{HM}_i$ outcome is optimal. However, determining the ``next best matching" with respect to the preference index is paramount for a district wishing to minimize priority violations. We can do so by listing assignments in order of increasing preference index (cf.\ \cite{Mu68}). Because the SOSM outcome is the lowest preference index stable matching, it will be the first stable matching found in this list.

Alternatively, one may start with the $\mathcal{HM}_i$ output and apply cycle improvements in the spirit of \cite{ErEr08}. Even though in \cite{ErEr08} a  \emph{stable cycle improvement} is only defined for stable starting points, it is possible to develop a modification of the concept to include starting situations which are unstable and aim toward a more stable assignment. Further analysis of this is beyond the scope of this note, but in such a scenario, each improvement would correspond to a decrease in priority violations, and thus produce a more stable matching.  Counting the priority violations, policy makers can determine what a particular student's priority is worth in terms of others' preferences.

An important observation to make is that in basing our mechanism on the preference index, we are merely attempting to refine a (rather coarse) partial order. This is not to say that the $\mathcal{HM}_i$ output will Pareto dominate matchings obtained via all previous mechanisms. Certainly the $\mathcal{HM}_i$ will not Pareto dominate TTC, given that TTC is Pareto efficient. Similarly, when the SOSM outcome is Pareto efficient, $\mathcal{HM}_i$ will not Pareto dominate it.


\subsection{An implementation issue: multiple minima}
\label{SS:ImplementationMultMin}

We saw that the ordering induced by the preference index is not strict. Indeed a given preference profile might have multiple minimum preference index solutions. 
The underlying theoretical problem of finding all possible minimum cost assignments by the Hungarian algorithm was addressed and answered in \cite{FuTo92} (see \cite{Fu94} for an improvement on the main (polynomial time) algorithm used in \cite{FuTo92} and \cite{MaPl05} for more recent work in a similar vein).  Thus it is possible to find all minimum preference index solutions. This in turn raises a new question:
How does one choose among all these minimum preference index solutions? We propose two possible approaches to deal with this issue.

\begin{enumerate}
\item If one intends to promote fairness by narrowing the discrepancies between the rankings of student assignments, then the matching with the minimum variance across individual student preference indices should be chosen. 
\item  If one intends to choose ``the most stable" matching, then the matching with the fewest instances of priority violations should be chosen. Here, we look at the number of students who had their priorities violated. 
\end{enumerate}

Of course, one may use both of these in succession.

This incidentally helps us resolve a possible concern about the Hungarian algorithm: its dependence on the order of the rows and the columns of the input matrix. Especially when there are multiple minimal index solutions, the order in which students or schools are listed may indeed affect the outcome, and the output matching may be different in different cases (though any two outcomes in such a scenario will have the same index). However if we modify our mechanism to look instead for all possible minimum cost matchings, this no longer creates a problem.
Thus, the order of the rows or columns ultimately doesn't matter because: (1) If there is a unique cost minimizing solution, the order does not affect the outcome; and (2) if there are multiple cost minimizing solutions, we can  find all of them using our mechanism, with  adaptations a la \cite{FuTo92}.


\section{Conclusion}
\label{S:Conclusion}

Current school choice mechanisms focus on balancing student preferences and school priorities, and the resulting matches sacrifice desirable characteristics.Ê Since a good public education is a scarce resource, there is no way to assign students to schools in such a way that all students attend top schools.Ê Much of the controversy in school choice concerns the ability of parents to send their children to their school of choice.Ê Thus in our approach we chose to focus on student preferences.Ê To this end, we developed a new criterion, or measure, by which to evaluate the quality of a matching that results from a school choice mechanism.Ê
The new mechanism presented here was adapted from the well known Hungarian algorithm \cite{Ku55} which was developed as a combinatorial solution to the assignment problem.Ê Our modifications included a re-interpretation of assignments taking into account school capacities and required that we be allowed to ``complete'' submitted student preference profiles. 

Our focus on student preferences over school priorities is natural in the current climate in which the public debate over charter schools and school vouchers rages in an attempt to offer parents more control over their children's educational choices. The fact that stability often comes at a loss in efficiency has been discussed at length in the School Choice literature. However, there might be cases in which, if we were to allow for simply one student with a violated priority, then we could make considerable efficiency gains (in terms of lower preference index).Ê Our method allows for viewing ``degrees of stability'' (e.g. ``highly stable''  would correspond to a ``low number'' of priority violations) in light of efficiency gains (as determined by the preference index). In cases where it is possible to achieve large efficiency gains while remaining highly or semi-stable, it might be advantageous to do so through our Hungarian Mechanism. Our method should also appeal to families, since their preferences are taken into account ``first.''Ê As assignments shift in an attempt to minimize preference violations and/or priority violations, some balance might be reached.


\bibliographystyle{aaai}   

\bibliography{SchoolChoiceBibliography}

\begin{thebibliography}{}

\bibitem[\protect\citeauthoryear{Abdulkadiro\v{g}lu and
  S\"{o}nmez}{2003}]{AbSo03}
Abdulkadiro\v{g}lu, A., and S\"{o}nmez, T.
\newblock 2003.
\newblock School choice: A mechanism design approach.
\newblock {\em The American Economic Review} 93(3):729--747.

\bibitem[\protect\citeauthoryear{Abdulkadiro\v{g}lu \bgroup et al\mbox.\egroup
  }{2005}]{APRS05}
Abdulkadiro\v{g}lu, A.; Pathak, P.~A.; Roth, A.~E.; and S\"{o}nmez, T.
\newblock 2005.
\newblock The boston public school match.
\newblock {\em American Economic Review - Papers and Proceedings}  368--371.

\bibitem[\protect\citeauthoryear{Abdulkadiro\v{g}lu \bgroup et al\mbox.\egroup
  }{2006}]{AbPaRoSo06}
Abdulkadiro\v{g}lu, A.; Pathak, P.~A.; Roth, A.~E.; and S\"{o}nmez, T.
\newblock 2006.
\newblock Changing the boston school-choice mechanism: Strategy-proofness as
  equal access.
\newblock {\em Working paper}.

\bibitem[\protect\citeauthoryear{Abdulkadiro\v{g}lu, Pathak, and
  Roth}{2005}]{APR05}
Abdulkadiro\v{g}lu, A.; Pathak, P.~A.; and Roth, A.~E.
\newblock 2005.
\newblock The new york city high school match.
\newblock {\em American Economic Review - Papers and Proceedings}  364--367.

\bibitem[\protect\citeauthoryear{Abdulkadiro\v{g}lu, Pathak, and
  Roth}{2009}]{AbPaRo09}
Abdulkadiro\v{g}lu, A.; Pathak, P.~A.; and Roth, A.~E.
\newblock 2009.
\newblock Strategy-proofness versus efficiency in matching with indifferences:
  Redesigning the nyc high school match.
\newblock {\em American Economic Review} 99(5):1954--1978.

\bibitem[\protect\citeauthoryear{Brown and Knight}{2005}]{BrKn05}
Brown, A.~K., and Knight, K.~W.
\newblock 2005.
\newblock School boundary and student assignment procedures in large, urban,
  public school systems.
\newblock {\em Education and Urban Society} 37:398--418.

\bibitem[\protect\citeauthoryear{Caro \bgroup et al\mbox.\egroup
  }{2004}]{CaShGuWe04}
Caro, F.; Shirabe, T.; Guignard, M.; and Weintraub, A.
\newblock 2004.
\newblock School redistricting: Embedding gis tools with integer programming.
\newblock {\em The Journal of the Operational Research Society} 55(8):836--849.

\bibitem[\protect\citeauthoryear{Cullen, Jacob, and Levitt}{2006}]{CuJaLe06}
Cullen, J.~B.; Jacob, B.~A.; and Levitt, S.
\newblock 2006.
\newblock The effect of school choice on participants: Evidence from randomized
  lotteries.
\newblock {\em Econometrica} 74(5):1191--1230.

\bibitem[\protect\citeauthoryear{Erdil and Ergin}{2008}]{ErEr08}
Erdil, A., and Ergin, H.~I.
\newblock 2008.
\newblock What's the matter with tie-breaking? improving efficiency in school
  choice.
\newblock {\em American Economic Review} 98(3):669--689.

\bibitem[\protect\citeauthoryear{Ferland and Guenett}{1990}]{FeGu90}
Ferland, J., and Guenett, G.
\newblock 1990.
\newblock Decision support system for the school districting problem.
\newblock {\em Operations Research} 38(1):15--20.

\bibitem[\protect\citeauthoryear{Franklin and Koenigsber}{1973}]{FrKo73}
Franklin, A., and Koenigsber, E.
\newblock 1973.
\newblock Computed school assignments in a large district.
\newblock {\em Operations Research} 21(2):413--426.

\bibitem[\protect\citeauthoryear{Fukuda and Matsui}{1992}]{FuTo92}
Fukuda, K., and Matsui, T.
\newblock 1992.
\newblock Finding all minimum-cost perfect matchings in bipartite graphs.
\newblock {\em Networks} 22(5):461--468.

\bibitem[\protect\citeauthoryear{Fukuda}{1994}]{Fu94}
Fukuda, K.
\newblock 1994.
\newblock Finding all the perfect matchings in bipartite graphs.
\newblock {\em Applied Mathematics Letters} 7(1):15--18.

\bibitem[\protect\citeauthoryear{Gale and Shapley}{1962}]{GaSh62}
Gale, D., and Shapley, L.~S.
\newblock 1962.
\newblock College admissions and the stability of marriage.
\newblock {\em The American Mathematical Monthly} 69(1):9--15.

\bibitem[\protect\citeauthoryear{Greene, Peterson, and Du}{1999}]{GrPeDu99}
Greene, J.~P.; Peterson, P.~E.; and Du, J.
\newblock 1999.
\newblock Effectiveness of school choice: The milwaukee experiment.
\newblock {\em Education and Urban Society} 31(2):190--213.

\bibitem[\protect\citeauthoryear{Kesten}{2010}]{Ke10}
Kesten, O.
\newblock 2010.
\newblock School choice with consent.
\newblock {\em Quarterly Journal of Economics} 125(3):1297--1348.

\bibitem[\protect\citeauthoryear{Kuhn}{1955}]{Ku55}
Kuhn, H.~W.
\newblock 1955.
\newblock The hungarian method for the assignment problem.
\newblock {\em Naval Research Logistics Quarterly} 2:83--97.

\bibitem[\protect\citeauthoryear{Kuhn}{1956}]{Ku56}
Kuhn, H.~W.
\newblock 1956.
\newblock Variants of the hungarian method for assignment problems.
\newblock {\em Naval Research Logistics Quarterly} 3:253--258.

\bibitem[\protect\citeauthoryear{Manea and Ploscaru}{2005}]{MaPl05}
Manea, F., and Ploscaru, C.
\newblock 2005.
\newblock Solving a combinatorial problem with network flows.
\newblock {\em Journal of Applied Mathematics and Computing} 17(1-2):391--399.

\bibitem[\protect\citeauthoryear{Maskin}{2008}]{Ma08}
Maskin, E.~S.
\newblock 2008.
\newblock Mechanism design: How to implement social goals.
\newblock {\em American Economic Review} 98(3):567--576.

\bibitem[\protect\citeauthoryear{Mcfadden}{2009}]{Mc09}
Mcfadden, D.
\newblock 2009.
\newblock The human side of mechanism design: a tribute to leo hurwicz and
  jean-jacque laffont.
\newblock {\em Rev Econ Design} 13(1-2):77--100.

\bibitem[\protect\citeauthoryear{Munkres}{1957}]{Mu57}
Munkres, J.
\newblock 1957.
\newblock Algorithms for the assignment and transportation problems.
\newblock {\em Journal of the Society for Industrial and Applied Mathematics}
  5(1):32--38.

\bibitem[\protect\citeauthoryear{Murty}{1968}]{Mu68}
Murty, K.
\newblock 1968.
\newblock An algorithm for ranking all the assignments in order of increasing
  cost.
\newblock {\em Operations Research} 16(3):682--687.

\bibitem[\protect\citeauthoryear{Rorres and Anton}{1984}]{RoAn84}
Rorres, C., and Anton, H.
\newblock 1984.
\newblock {\em Applications of Linear Algebra}.
\newblock John Wiley and Sons, 3rd edition edition.

\bibitem[\protect\citeauthoryear{Roth and Sotomayor}{1990}]{RoSo90}
Roth, A.~E., and Sotomayor, M.
\newblock 1990.
\newblock {\em Two-Sided Matching: A Study in Game-Theoretic Modeling and
  Analysis}.
\newblock Econometric Society Monograph Series. Cambridge University Press.

\bibitem[\protect\citeauthoryear{Roth}{1982}]{Ro82}
Roth, A.~E.
\newblock 1982.
\newblock The economics of matching: Stability and incentives.
\newblock {\em Mathematics of Operations Research} 7(4):617--628.

\bibitem[\protect\citeauthoryear{Roth}{2008}]{Ro08}
Roth, A.~E.
\newblock 2008.
\newblock Deferred acceptance algorithms: History, theory, practice, and open
  questions.
\newblock {\em International Journal of Game Theory} 36:537--569.

\bibitem[\protect\citeauthoryear{Rouse}{1998}]{Ro98}
Rouse, C.~E.
\newblock 1998.
\newblock Schools and student achievement: More evidence from the milwaukee
  parental choice program.
\newblock {\em FRBNY Economic Policy Review}  61--76.

\bibitem[\protect\citeauthoryear{Schrijver}{2003}]{Sc03}
Schrijver, A.
\newblock 2003.
\newblock {\em Combinatorial optimization. {P}olyhedra and efficiency. {V}ol.
  {A}}, volume~24 of {\em Algorithms and Combinatorics}.
\newblock Berlin: Springer-Verlag.
\newblock Paths, flows, matchings, Chapters 1--38.

\bibitem[\protect\citeauthoryear{Tadenuma}{2002}]{Ta02}
Tadenuma, K.
\newblock 2002.
\newblock Efficiency first or equity first? two principles and rationality of
  social choice.
\newblock {\em Journal of Economic Theory} 104(2):462--472.

\end{thebibliography}


\end{document}